\documentclass{conm-p-l}

\newcommand{\abs}[1]{\lvert#1\rvert}

\newtheorem{theorem}{Theorem}[section]

\theoremstyle{definition}

\theoremstyle{remark}

\numberwithin{equation}{section}

\begin{document}

\title{Local dependencies in random fields via a Bonferroni-type inequality}

\author{Adam Jakubowski}
\address{Faculty of Mathematics and Computer Science, Nicholas Copernicus
University, Chopina 12/18, 87-100 Toru\'n, Poland}
\email{adjakubo@mat.uni.torun.pl}
\thanks{Research of the first author was done during visits to
Universit\'e de Lille I and University of Tennessee, Knoxville.}

\author{Jan Rosi\'nski}
\address{Department of Mathematics, University of Tennessee, Knoxville, TN
37996, USA}
\email{rosinski@math.utk.edu}
\thanks{Research of the second author was supported in part by the NSF
Grant DMS-97-04744.}

\date{}

\begin{abstract}
We provide an inequality which is a useful tool in studying both large
deviation results and limit theorems for sums of random fields with
"negligible" small values. In particular, the inequality covers cases of
stable limits for random  
variables with heavy tails and compound Poisson limits of $0-1$ random
variables. 
\end{abstract}

\maketitle

\section{Bonferroni-type inequalities in limit theorems for sums of
stationary sequences}

The simplest Bonferroni-type inequality can be formulated in the following
way (see
inequality I.17, p. 16, \cite{GS96}):
\begin{equation}\label{e1}
0 \leq \sum_{i=1}^n P(A_i)- P(\bigcup_{i=1}^n A_i) \leq \sum_{1 \leq i < j
\leq n} P(A_i\cap A_j),
\end{equation}
where $A_1, A_2, \ldots, A_n$ are events in some probability space. 

In general this inequality gives very bad estimate for the difference 
$\sum_{i=1}^n P(A_i)- P(\bigcup_{i=1}^n A_i)$
(see p. 19, \cite{GS96} for
discussion of typical examples). However, when properly used, it brings
essential simplification in many areas.
Perhaps the most known (and the simplest) is the limit theory for order
statistics of stationary sequences, as presented in \cite{LLR83} or
\cite{G78}. It may be instructive to provide the reader with a brief
outline of the reasoning leading to the basic result of this theory
(Theorem 3.4.1, Chapter 3, \cite{LLR83}). 

Let $X_1, X_2, \ldots,$ be a
stationary sequence and let $M_n = \max_{1\leq i \leq n} X_i$ be partial
maxima for this sequence. Given a sequence $\{u_n\}$ of numbers we want to
calculate the limit for $P(M_n \leq u_n)$.  For a large class of
stationary sequences 
(satisfying so called condition $D(u_n)$), we can asymptotically replace
$P(M_n \leq u_n)$ with 
\[ P(M_{[n/k_n]} \leq u_n)^{k_n},\]
with some $k_n \to \infty$. This in turn is asymptotically the same as
\[ \exp(-k_n P(M_{[n/k_n]} > u_n)).\]
For fixed $n$, set $A_i = \{X_i > u_n\}$ and observe that by (\ref{e1})
\[ 
\begin{split}
k_n \abs{P(M_{[n/k_n]} > u_n) - &[n/k_n]P(X_1 > u_n)}\\
 &\leq k_n \sum_{1
\leq i 
< j \leq [n/k_n]} P(X_i > u_n, X_j > u_n).
\end{split}
\]
If so called condition $D'(u_n)$ is also satisfied, then the last
expression above tends to zero as $n \to \infty$ and we can
calculate the 
limit for $P(M_n \leq u_n)$ as if the random variables $X_i$ were
independent, i.e. 
\[ \lim_{n\to\infty} P(M_n \leq u_n) = \exp (- \lim_{n\to\infty} n P(X_1 >
u_n)).\]  
Condition $D(u_n)$ represents here ``mixing" or ``weak dependence"
properties of the sequence in the form proper
for maxima, while condition $D'(u_n)$ asserts that in the sequence
$\{X_i\}$ there are no local (within intervals of length $[n/k_n]$)
clusters of values exceeding levels $u_n$. Since independent random
variables satisfy condition $D'(u_n)$ for sequences $\{u_n\}$ of interest,
one can also say that
the sequence essentially has no ``local dependencies" between random variables.
The latter terminology is even more convincing when one realizes that
condition $D'(u_n)$ cannot hold for $1$-dependent random
variables $X_i=Y_{i-1}\vee Y_i$, where
$Y_i$ is a sequence of independent and identically distributed random
variables and $u_n$ is such that
$\liminf_n nP(Y_1 > u_n) > 0$. Clearly, such $X_i$'s exhibit ``local
dependencies" and admit ``local clusters" of values exceeding levels $u_n$.

It was R.A. Davis who first observed that similar results hold also for
sums of stationary sequences with heavy tails. Using the technique of
extreme value theory as well as the series representation for stable laws
due to LePage, Woodroofe and Zinn, Davis \cite{D83} proved that
asymptotics of sums of 
``weakly dependent" stationary random variables {\em with} heavy tails and
{\em without} local dependencies
is essentially the same as if they were independent. Subsequent papers
\cite{JK89}, \cite{DH95}, \cite{K95}  showed that the essence of Davis'
method was representing sums as integrals with respect
to point processes on $\mathbb{R}^1\setminus
\{0\}$ built upon the sequence $X_i$.  If one defines $N_n(A) =
\sum_{i=1}^n I(X_i/B_n \in A)$, then 
\[ \frac{X_1 + X_2 + \ldots + X_n}{B_n} = \int_{\mathbb{R}^1 \setminus
\{0\}} x N_n(dx),\]
and weak convergence of $N_n$'s implies weak convergence of $S_n/B_n$. In
particular, results for sums of dependent sequences with heavy tails can
be obtained in a similar way as results for sums of independent sequences
were derived in \cite{R86} (this analogy is not applicable for functional
convergence). 

The difference between weakly dependent and independent case is
that in the absence of conditions excluding clusters of ``big" values
(like $D'(u_n)$ in the theory for extremes), the parameters of the
limiting stable law are determined by local dependence properties. Davis
and Hsing \cite{DH95} provide a probabilistic
representation for these parameters.
In some cases (e.g. for $m$-dependent random variables) another, much
simpler representation is available \cite{JK89}, which is valid also for
generalized Poisson limits \cite{K95}.
Comparing to stable limit theorems for $m$-dependent random vectors
obtained by purely 
analytical methods by L. Heinrich 
in \cite{H82}, \cite{H85}, probabilistic reasoning gave both deeper
insight into the structure of the limiting stable laws and allowed
avoiding many of technicalities in formulation of results.
When specialized to sums of $m$-dependent $0-1$ random
variables, the point processes method provides sufficient and {\em
necessary} conditions for 
convergence to compound Poisson distribution \cite{K95}, contrary to the
earlier 
methods based on Poisson 
approximations via the Chen-Stein method (see e.g. \cite{AGG90}), where
only sufficient conditions are given.

It is interesting that most of the above results can be obtained without
employing point processes techniques and  using the following
Bonferroni-type inequality.

\begin{theorem}[Lemma 3.2, \cite{J97}]
Let $Z_1, Z_2, \ldots $ be stationary random vectors taking values in
a linear space $(E, \mathcal{B}_E)$. Set $S_0 = 0$, $S_k = \sum_{j=1}^k Z_j,\
k\in\mathbb{N}$.

If $U \in \mathcal{B}_E$ is such that $0 \notin U$, then for every $n
\in \mathbb{N}$ and every $m$, $0 \leq m \leq n$, the following
inequality holds:
\begin{equation}\label{emain}
\begin{split}
| P( S_n \in U) &- n\big( P(S_{m+1} \in U) - P(S_m \in U)\big)|
\vphantom{\sum_{1 \leq i < j \leq n}}\qquad\qquad \\
&\leq 
2 m P(Z_1 \neq 0) + 2 \sum_{\genfrac{}{}{0pt}{2}{1 \leq i < j \leq n}{j -
i > m}} 
P(Z_i \neq 0, Z_j \neq 0).
\end{split}
\end{equation}
\end{theorem}

Although inequality (\ref{emain}) does not fit the formal definition of the
Bonferroni-type inequality given on p. 10 in \cite{GS96}, we call it
Bonferroni-type for the following reasons.
\begin{enumerate}
\item When $m=0$ we obtain from (\ref{emain})
\[| P(S_n \in U) - n P(Z_1 \in U)| \leq  
2 \sum_{1 \leq i < j \leq n}
P(Z_i \neq 0, Z_j \neq 0), \]
what is formally similar to (\ref{e1}). Notice that the constant 2 above
is sharp.
\item The inequality becomes interesting only if we deal with at least
``weak dependence", that is under purely probabilistic assumption.
\item The inequality is proved by integrating its pointwise version  and
in this sense its proof is similar to proofs of the Bonferroni-type 
inequalities obtained by the ``indicator method" (see \cite{GS96}).
\item In Section 3 we provide a unifying framework for both inequalities
(\ref{e1}) and (\ref{emain}). 
\end{enumerate}

The inequality looks very restrictive and may seem applicable only to 0-1
stationary random variables $Z_j = I_{A_j}$, in which case it reads as
follows.  
\[
\begin{split}
\bigg| P( \sum_{j=1}^{n} I_{A_j} = k) - n\bigg( P(\sum_{j=1}^{m+1}
I_{A_j} = k) 
&- P(\sum_{j=1}^{m} I_{A_j} = k)\bigg)\bigg| 
\vphantom{\sum_{1 \leq i < j \leq n}}\qquad\qquad \\
&\leq 
2 m P(A_1) + 2 \sum_{\genfrac{}{}{0pt}{2}{1 \leq i < j \leq n}{j -
i > m}} 
P(A_i \cap A_j).
\end{split}
\]
The above inequality can be directly applied to give an alternative (and
much simpler!) proof
of results due to Kobus \cite{K95} for $m$-dependent $0-1$ random variables.

Originally however inequality (\ref{emain}) was designed to manipulate with
probabilities of large deviation for sums of random variables with heavy
tails. An extensive discussion of such results as well as their meaning for
stable limit theorems (essential part of necessary and sufficient
conditions) can be found in \cite{J93}, \cite{J97} and \cite{JNZ97} (for
necessary results on stable laws we refer to \cite{JW94} and \cite{ST94}).
Here let us sketch basic ideas only.

Let $X_1, X_2, \ldots$ be a stationary sequence, $S_n = X_1 + X_2 + \ldots
+ X_n$, $B_n \to \infty$ be a
$1/p$-regularly varying sequence, where $ 0< p < 2$  and let $x_n \to \infty$.
We are interested in asymptotic behavior of large deviation probabilities
$P(S_n/B_n > x_n)$. More precisely, we want to prove that under relatively
mild assumptions
\begin{equation}
\label{LD}
 x_n^p P(S_n/B_n > x_n) \to c_+,\end{equation}
where the constant $0 < c_+ < \infty$ can be identified. The first step
consists in proving that as $n \to \infty$
\[ x_n^p \big(P(S_n/B_n > x_n) - P(\sum_{j=1}^n Z_{n,j}^{\delta_n} > x_n)
\big) \to 0,\]
where 
\[ Z_{n,j}^{\delta_n} = \begin{cases}
0 & \text{if $|X_j| < B_n \cdot x_n \cdot \delta_n$},\\
X_j/B_n & \text{otherwise}.
\end{cases}
\]
This requires some polynomial domination condition on
tail probabilities of $X_j$'s and, if $1 \leq p < 2$, some assumptions on
the size of variances of random variables $T_n^{\delta_n} = S_n/B_n -
\sum_{j=1}^n Z_{n,j}^{\delta_n}$. 

In the next, essential step, we apply 
inequality (\ref{emain}) to random variables $Z_{n,j}^{\delta_n}$,
$j=1,2,\ldots, n$, and $U = (x_n, \infty)$. Careful control of the size
of $x_n$ and $\delta_n$ \textit{plus} information on dependence (e.g.
$m$-dependence) \textit{plus} return to original random variables allow
reducing (\ref{LD}) to 
\[ x_n^p n \big(P( X_1 + X_2 + \ldots + X_{m+1} > x_n B_n) - P(X_1 + X_2
\ldots + 
X_m > x_n B_n)\big) \to c_+.\]
This shows that the limiting parameter $c_+$ can be calculated using 
only finite dimensional (of size $m+1$) distributions of the sequence
$X_1, X_2, \ldots, $ and that its value depends on local dependence
structures, as desired.

Clearly, variants of the above reasoning with $m$ varying are also
workable.

In the present paper we are going to prove an analog of (\ref{emain})
for random fields and in nonstationary case. Following the line of
\cite{J97} it allows deriving results for $m$-dependent random fields,
similar to stable limit theorems of Heinrich \cite{H86}, \cite{H87} or
results on convergence to compound Poisson distributions \cite{AGG90}.  We
leave their extensive discussion to 
other place.

\section{A Bonferroni-type inequality for random fields}

In what follows we choose and fix two integer numbers: \\
$d$ - the dimension of the
lattice $\mathbb{Z}^d$ indexing random fields
$\{Z_{\boldsymbol{t}}\}_{\boldsymbol{t}\in \mathbb{Z}^d}$;\\
$m$ - the admissible size of local clusters, $m \geq 0$.

If $\{Z_{\boldsymbol{t}}\}$
is a random field 
and  $\Lambda \subset \mathbb{Z}^d$ is a finite set, we define
\begin{equation}\label{eslambda}
S_{\Lambda} = \sum_{\boldsymbol{t}\in \Lambda} Z_{\boldsymbol{t}},\ \ 
S_{\emptyset} = 0.
\end{equation}
Let
\[B = \{0,1,\ldots,m\}^d,\]
and $B_{\boldsymbol{t}} = B + \boldsymbol{t}$, \ 
$\boldsymbol{t} \in \mathbb{Z}^d$.  Further, let
\[\mathcal{E} = \{0,1\}^d = \{\boldsymbol{\varepsilon} = (\varepsilon_1,
\varepsilon_2, \ldots, \varepsilon_d):\varepsilon_j = 0\ \text{or}\ 1\}\]
and let
\begin{equation} \label{Be}
B_{\boldsymbol{t}}^{\boldsymbol{\varepsilon}} = B_{\boldsymbol{t}}\cap
B_{\boldsymbol{t + \varepsilon}}, \ \ \ \boldsymbol{\varepsilon} \in
\mathcal{E}.
\end{equation}
Define, for $U\in \mathcal{B}_E$ and $\boldsymbol{t}\in \mathbb{Z}^d$,
\[ \Delta_{\boldsymbol{t}}(U) = \sum_{\boldsymbol{\varepsilon} \in
\mathcal{E}}
(-1)^{\abs{\boldsymbol{\varepsilon}}}
P(S_{B_{\boldsymbol{t}}^{\boldsymbol{\varepsilon}}} 
\in U)\]
where 
\[ \abs{\boldsymbol{\varepsilon}} = \varepsilon_1 + 
\varepsilon_2 +  \cdots +\varepsilon_d.\]
Put $\boldsymbol{1} = (1,\ldots,1) \in \mathcal{E}$. Define the
{\it ``boundary''} of a  set $\Lambda \subset \mathbb{Z}^d$ by
\begin{equation}\label{partl}
\partial\Lambda = \{\boldsymbol{s} \notin \Lambda :
\exists_{\boldsymbol{t} \in \Lambda}\ \boldsymbol{s} \in 
B_{\boldsymbol{t}}\} \cup  \{\boldsymbol{t} \in \Lambda :
\exists_{\boldsymbol{s} \in \Lambda^c}\ \boldsymbol{t} \in 
B_{\boldsymbol{s}}\setminus B_{\boldsymbol{s+\boldsymbol{1}}}\}
\end{equation}
Notice that the second part of $\partial\Lambda$,
consisting of points from $\Lambda$, is empty when $d=1$.
\vskip18pt

\begin{theorem}\label{TM}
Let $Z_{\boldsymbol{t}}$, $\boldsymbol{t}\in \mathbb{Z}^d$ be a random
field with values in a linear space $(E, \mathcal{B}_E)$. If $U \in
\mathcal{B}_E$ and $0 \notin U$ then 
\vskip6pt
\begin{equation}\label{eM}
\begin{split}
\abs{P(S_{\Lambda} \in U) - \sum_{\boldsymbol{t} \in \Lambda}
\Delta_{\boldsymbol{t}}(U)} &\leq  c_1(d,m) \sum_{\boldsymbol{s}
\in \partial\Lambda} P(Z_{\boldsymbol{s}} \neq 0) \\
&\quad + c_2(d,m)
\sum_{\genfrac{}{}{0pt}{1}{\boldsymbol{s,t} \in \Lambda}{\|\boldsymbol{t} -
\boldsymbol{s}\|_{\infty} > m}} P(Z_{\boldsymbol{s}} \neq
0,Z_{\boldsymbol{t}} \neq 0), 
\end{split}
\end{equation}
\vskip6pt \noindent
where 
$ c_1(d,m) = 2^{d}((m+1)^d  - 1),$
and
$ c_2(d,m) = 2^{-1} (1 + 2^{d} (2m+1)^d).$
\end{theorem}
\vskip18pt

\begin{proof} 
Let
\begin{equation}\label{edelta}
\delta_{\boldsymbol{t}}(U) = \sum_{\boldsymbol{\varepsilon} \in
\mathcal{E}}
(-1)^{\abs{\boldsymbol{\varepsilon}}}
I(S_{B_{\boldsymbol{t}}^{\boldsymbol{\varepsilon}}} 
\in U).
\end{equation}
Since $\Delta_{\boldsymbol{t}}(U) = E \delta_{\boldsymbol{t}}(U)$,
it is enough
to establish a ``pointwise" version of (\ref{eM}), i.e. 
\begin{equation}\label{eM1}
\begin{split}
\abs{I(S_{\Lambda} \in U) - \sum_{\boldsymbol{t} \in \Lambda}
\delta_{\boldsymbol{t}}(U)} &\leq  c_1(d,m) \sum_{\boldsymbol{s}
\in \partial\Lambda} I(Z_{\boldsymbol{s}} \neq 0) \\
&\quad + c_2(d,m)
\sum_{\genfrac{}{}{0pt}{1}{\boldsymbol{s,t} \in \Lambda}{\|\boldsymbol{t} -
\boldsymbol{s}\|_{\infty} > m}} I(Z_{\boldsymbol{s}} \neq
0,Z_{\boldsymbol{t}} \neq 0), 
\end{split}
\end{equation}

We shall deal with a modification of $Z_{\boldsymbol{t}}$ which vanishes
outside our set $\Lambda$:
\[
Z_{\boldsymbol{t}}' = \begin{cases}
Z_{\boldsymbol{t}} & \text{if $\boldsymbol{t}\in \Lambda$},\\
0 & \text{if $\boldsymbol{t} \notin \Lambda$}.
\end{cases}
\]
Let $S_{\Lambda}'$ and $\delta_{\boldsymbol{t}}'(U)$ denote quantities
defined by replacement of $Z_{\boldsymbol{t}}$ with $Z_{\boldsymbol{t}}'$
in formulas (\ref{eslambda}) and (\ref{edelta}), respectively. Then by the
very definition we
have 
\[ I(S_{\Lambda}\in U) = I(S_{\Lambda}'\in U). \]
Further,
$\delta_{\boldsymbol{t}}(U) \neq \delta_{\boldsymbol{t}}'(U)$
implies that there exists $\boldsymbol{s} \in B_{\boldsymbol{t}} \cap
\Lambda^c$ 
such that $Z_{\boldsymbol{s}} \neq 0$.
Hence we can estimate
\begin{equation}\label{eM2}
\begin{split}
\abs{(I(S_{\Lambda} \in U) - \sum_{\boldsymbol{t} \in \Lambda}
\delta_{\boldsymbol{t}}(U)) &- (I(S_{\Lambda}' \in U) - \sum_{\boldsymbol{t}
\in \Lambda} 
\delta_{\boldsymbol{t}}'(U))}\\ &\leq
\sum_{\genfrac{}{}{0pt}{1}{\boldsymbol{t} \in
\Lambda}{B_{\boldsymbol{t}}\cap \Lambda^c \neq \emptyset}} 
2^{d} I(\exists_{\boldsymbol{s}\in B_{\boldsymbol{t}}\cap \Lambda^c}
Z_{\boldsymbol{s}} \neq 0)\\
&\leq
2^{d} \sum_{\genfrac{}{}{0pt}{1}{\boldsymbol{t} \in
\Lambda}{B_{\boldsymbol{t}}\cap \Lambda^c \neq \emptyset}}
 \sum_{\boldsymbol{s}\in B_{\boldsymbol{t}}\cap \Lambda^c}
I(Z_{\boldsymbol{s}} \neq 0) =: R_1,
\end{split}
\end{equation}
where the factor $2^{d}$ comes from the cardinality of $\mathcal{E}$.
Furthermore, if $\partial_1\Lambda$ denotes the first part of the boundary
(\ref{partl}) consisting of points from $\Lambda^c$, then
\begin{equation}\label{eM3}
\begin{split}
R_1 &= 2^{d} \sum_{\boldsymbol{t} \in \Lambda} \sum_{\boldsymbol{s}\in
\Lambda^c} I_{B_{\boldsymbol{t}}\cap 
\Lambda^c}(\boldsymbol{s}) I(Z_{\boldsymbol{s}} \neq 0) \\
 &= 2^{d} \sum_{\genfrac{}{}{0pt}{1}{\boldsymbol{s} \in
\Lambda^c}{\exists_{\boldsymbol{t}\in \Lambda} \boldsymbol{s} \in
B_{\boldsymbol{t}}}}  
( \sum_{\boldsymbol{t}\in \Lambda} I_{B_{\boldsymbol{t}}\cap
\Lambda^c}(\boldsymbol{s}))
I(Z_{\boldsymbol{s}} \neq 0) \\
&\leq 2^{d} ((m+1)^d -1) \sum_{\boldsymbol{s} \in \partial_1\Lambda}
I(Z_{\boldsymbol{s}}\neq 0).
\end{split}
\end{equation}

Hence it suffices to prove (\ref{eM1}) under the assumption that
\begin{equation}\label{zerovalue}
Z_{\boldsymbol{t}} = 0\ \ \text{\rm for}\ \ \boldsymbol{t} \notin \Lambda.
\end{equation}
In this case the first sum on the right hand side of (\ref{eM1}) will be over
the second part of the boundary (\ref{partl}) consisting of points from
$\Lambda$.  Now define a random set
\[ \Lambda_0 = \{\boldsymbol{s}\in \mathbb{Z}^d : Z_{\boldsymbol{t}} \ne 0\}
\]
and let
\[ \text{diam}(\Lambda_0) = \sup\{\|\boldsymbol{s}-\boldsymbol{u}\|_{\infty} :
\boldsymbol{s},\boldsymbol{u}\in \Lambda_0\}.
\]
Notice that (\ref{zerovalue}) gives
\begin{equation} \label{sub}
\Lambda_0 \subset \Lambda
\end{equation}
so that \ $\text{diam}(\Lambda_0)$ \ is a bounded random variable.
For a \underline{fixed} point $\omega$ in the probability space we will
consider three particular cases of $\text{diam}(\Lambda_0(\omega))$:

\subsection*{Case 1. \ $\text{diam}(\Lambda_0)\le m$.} \ 

This assumption implies that $\Lambda_0 \subset B_{\boldsymbol{t}_0}$
for some $\boldsymbol{t}_0\in \mathbb{Z}^d$. Hence
\begin{equation}\label{eMxx0}
 I(S_{\Lambda} \in U) =I(S_{\Lambda_0} \in U)= I(S_{B_{\boldsymbol{t}_0}}
\in U).
\end{equation}
The first observation is that if $\boldsymbol{t} \notin
B_{\boldsymbol{t}_0}$, then $\delta_{\boldsymbol{t}}(U) = 0$ and
consequently
\begin{equation}\label{eMxx}
\sum_{\boldsymbol{t} \in \Lambda}
\delta_{\boldsymbol{t}}(U) = \sum_{\boldsymbol{t} \in B_{\boldsymbol{t}_0}
\cap \Lambda}
\delta_{\boldsymbol{t}}(U).
\end{equation}
Indeed, for $\boldsymbol{t} = (t_1,t_2,\ldots,t_d) \notin
B_{\boldsymbol{t}_0}$ we have
\[
\begin{split}
\delta_{\boldsymbol{t}}(U) &= \sum_{\boldsymbol{\varepsilon} \in
\mathcal{E}}
(-1)^{\abs{\boldsymbol{\varepsilon}}}
I(S_{B_{\boldsymbol{t}}^{\boldsymbol{\varepsilon}}} 
\in U)\\
&=  \sum_{\boldsymbol{\varepsilon} \in
\mathcal{E}}
(-1)^{\abs{\boldsymbol{\varepsilon}}}
I(S_{B_{\boldsymbol{t}}\cap B_{\boldsymbol{t + \varepsilon}} \cap
B_{\boldsymbol{t_0}}}  
\in U).
\end{split}
\]
If $t_k \geq t^0_k,\ k=1,2,\ldots,d$, where $\boldsymbol{t}_0 =
(t^0_1,t^0_2,\ldots,t^0_d)$, then either $\boldsymbol{t} \in
B_{\boldsymbol{t}_0}$ or $B_{\boldsymbol{t}} \cap B_{\boldsymbol{t}_0} =
\emptyset$. The former possibility has been excluded and by the latter
$\delta_{\boldsymbol{t}}(U) = 0$. So assume that $t_k < t^0_k$ for some
$k,\ 1 \leq k \leq d$. Let $\boldsymbol{\varepsilon}' =
(\varepsilon_1,\ldots,\varepsilon_{k-1},0,\varepsilon_{k+1},\ldots,
\varepsilon_d)$ and $\boldsymbol{\varepsilon}'' =
(\varepsilon_1,\ldots,\varepsilon_{k-1},1,\varepsilon_{k+1},\ldots,
\varepsilon_d)$. Then 
\[ S_{B_{\boldsymbol{t}}\cap B_{\boldsymbol{t +
\varepsilon'}} \cap B_{\boldsymbol{t_0}}} = S_{B_{\boldsymbol{t}}\cap
B_{\boldsymbol{t + \varepsilon''}} \cap B_{\boldsymbol{t_0}}},\] 
hence
\[ (-1)^{\abs{\boldsymbol{\varepsilon}'}}
I(S_{B_{\boldsymbol{t}}\cap B_{\boldsymbol{t + \varepsilon'}} \cap
B_{\boldsymbol{t_0}}} \in U) + (-1)^{\abs{\boldsymbol{\varepsilon''}}}
I(S_{B_{\boldsymbol{t}}\cap B_{\boldsymbol{t + \varepsilon''}} \cap
B_{\boldsymbol{t_0}}} \in U) = 0,\]
and so $\delta_{\boldsymbol{t}}(U) = 0$. Thus (\ref{eMxx}) follows and now we
will prove that
\begin{equation}
I(S_{B_{\boldsymbol{t}_0}} \in U) = 
\sum_{\boldsymbol{t} \in B_{\boldsymbol{t}_0}}
\delta_{\boldsymbol{t}}(U).
\end{equation} 
Define 
\[q_{\boldsymbol{t}} =
I(S_{B_{\boldsymbol{t}}\cap B_{\boldsymbol{t}_0}}\in U)\] 
and for $A \subset
\mathbb{Z}^d$  
\[ Q(A) = \sum_{\boldsymbol{t} \in A\cap B_{\boldsymbol{t}_0}}
q_{\boldsymbol{t}}.\] 
Notice that for some $\boldsymbol{t} \in B_{\boldsymbol{t}_0}$ the points
$\boldsymbol{t+\varepsilon}$ may lie outside of $B_{\boldsymbol{t}_0}$, but
then
$q_{\boldsymbol{t+\varepsilon}}=0$ and so
\[ Q(B_{\boldsymbol{\varepsilon}}) = \sum_{\boldsymbol{t} \in
B_{\boldsymbol{t}_0}} q_{\boldsymbol{t + \varepsilon}}.\]
By the assumption currently in force
\begin{equation}\label{eMxxx}
\begin{split}
\sum_{\boldsymbol{t} \in B_{\boldsymbol{t}_0}}
\delta_{\boldsymbol{t}}(U) &= \sum_{\boldsymbol{t} \in B_{\boldsymbol{t}_0}}
\sum_{\boldsymbol{\varepsilon} \in
\mathcal{E}} (-1)^{\abs{\boldsymbol{\varepsilon}}}
I(S_{B_{\boldsymbol{t}}\cap B_{\boldsymbol{t + \varepsilon}} \cap
B_{\boldsymbol{t}_0}}  \in U)\\
&= \sum_{\boldsymbol{t} \in B_{\boldsymbol{t}_0}}
\sum_{\boldsymbol{\varepsilon} \in \mathcal{E}}
(-1)^{\abs{\boldsymbol{\varepsilon}}} q_{\boldsymbol{t + \varepsilon}}\\
&= \sum_{\boldsymbol{\varepsilon} \in \mathcal{E}}
(-1)^{\abs{\boldsymbol{\varepsilon}}} \sum_{\boldsymbol{t} \in
B_{\boldsymbol{t}_0}} q_{\boldsymbol{t + \varepsilon}}\\
&= \sum_{\boldsymbol{\varepsilon} \in \mathcal{E}}
(-1)^{\abs{\boldsymbol{\varepsilon}}} Q(B_{\boldsymbol{\varepsilon}})
\end{split}
\end{equation}
The function $Q$ is additive, hence by the inclusion-exclusion formula
\begin{equation}
\label{eMxxxx} Q(\bigcup_{k=1}^d B_{\boldsymbol{e}_k}) =
\sum_{\genfrac{}{}{0pt}{1}{\boldsymbol{\varepsilon} \in
\mathcal{E}}{\boldsymbol{\varepsilon} \neq \boldsymbol{0}}}
(-1)^{\abs{\boldsymbol{\varepsilon}} - 1} Q(\bigcap_{\{k : \varepsilon_k =
1\}} B_{\boldsymbol{e}_k})=
\sum_{\genfrac{}{}{0pt}{1}{\boldsymbol{\varepsilon} \in 
\mathcal{E}}{\boldsymbol{\varepsilon} \neq \boldsymbol{0}}}
(-1)^{\abs{\boldsymbol{\varepsilon}} - 1} Q(B_{\boldsymbol{\varepsilon}}),
\end{equation}
where $\boldsymbol{e}_k$ are the standard unit vectors in $\mathbb{Z}^d$.
Combining (\ref{eMxxx}) and (\ref{eMxxxx}) we obtain
\[ \sum_{\boldsymbol{t} \in B_{\boldsymbol{t}_0}}
\delta_{\boldsymbol{t}}(U) = Q(B_{\boldsymbol{0}}) - Q(\bigcup_{k=1}^d
B_{\boldsymbol{e}_k}) = q_{\boldsymbol{0}} = I(S_{B_{\boldsymbol{t}_0}}
\in U).\] 
Hence, in view of (\ref{eMxx0})--(\ref{eMxx}),
\[
I(S_{\Lambda} \in U) -\sum_{\boldsymbol{t} \in \Lambda}
\delta_{\boldsymbol{t}}(U) = \sum_{\boldsymbol{s} \in
B_{\boldsymbol{t}_0}\cap \Lambda^c}
\delta_{\boldsymbol{s}}(U).
\]
Observe that $\delta_{\boldsymbol{s}}(U)\ne 0$ implies that
there is $\boldsymbol{t}\in  B_{\boldsymbol{s}}\setminus
B_{\boldsymbol{s+\boldsymbol{1}}}$ such that $Z_{\boldsymbol{t}}\ne 0$.
Indeed, if this is not the case then 
$S_{B_{\boldsymbol{s}}^{\boldsymbol{\varepsilon}}}=
S_{B_{\boldsymbol{s}}^{\boldsymbol{1}}}$ for every $\boldsymbol{\varepsilon}
\in \mathcal{E}$, and 
\begin{equation} \label{b}
\delta_{\boldsymbol{s}}(U) = (\sum_{\boldsymbol{\varepsilon}
\in
\mathcal{E}}
(-1)^{\abs{\boldsymbol{\varepsilon}}})
I(S_{B_{\boldsymbol{s}}^{\boldsymbol{1}}}
\in U) = 0.
\end{equation}
Consequently,
\[
\begin{split}
| I(S_{\Lambda} \in U) -\sum_{\boldsymbol{t} \in \Lambda}
\delta_{\boldsymbol{t}}(U)| &\le 2^d\sum_{\boldsymbol{s} \in
 \Lambda^c} \sum_{\boldsymbol{t}\in \Lambda}
I(\boldsymbol{t} \in  B_{\boldsymbol{s}}\setminus
B_{\boldsymbol{s+\boldsymbol{1}}}) I(Z_{\boldsymbol{t}}\ne 0)\\
&\le 2^d((m+1)^d - m^d-1)\sum_{\boldsymbol{t}\in \partial\Lambda}
I(Z_{\boldsymbol{t}}\ne 0)
\end{split}
\]
which is clearly dominated by the right hand side of (\ref{eM1}).  
\vskip18pt

\subsection*{Case 2. $m < \text{diam}(\Lambda_0)\le 2m$.} \ 

In this case there exist $\boldsymbol{t}_0, \boldsymbol{s}_0 \in \Lambda$
such that $Z_{\boldsymbol{t}_0}\ne 0$, $Z_{\boldsymbol{s}_0}\ne 0$,
and $\|\boldsymbol{t}_0- \boldsymbol{s}_0\|_{\infty} > m$. Hence 
\begin{equation}
\label{jeden}
2^{-1}\sum_{\genfrac{}{}{0pt}{1}{\boldsymbol{s,t} \in
\Lambda}{\|\boldsymbol{t} - 
\boldsymbol{s}\|_{\infty} > m}} I(Z_{\boldsymbol{s}} \neq
0,Z_{\boldsymbol{t}} \neq 0) \geq 1
\end{equation}
and trivially,
\begin{equation}
\label{dwa} I(S_{\Lambda} \in U) \leq
2^{-1}\sum_{\genfrac{}{}{0pt}{1}{\boldsymbol{s,t} \in
\Lambda}{\|\boldsymbol{t} - 
\boldsymbol{s}\|_{\infty} > m}} I(Z_{\boldsymbol{s}} \neq
0,Z_{\boldsymbol{t}} \neq 0).
\end{equation}
By the present assumption there exists $\boldsymbol{t}_0$ such that
\[
\Lambda_0 \subset \{0,\ldots,2m\}^d + \boldsymbol{t}_0:= K_{\boldsymbol{t}_0}.
\]
Similarly as in Case 1 we argue that $\delta_{\boldsymbol{t}}(U)=0$ for 
$\boldsymbol{t}\notin K_{\boldsymbol{t}_0}$.
Since each such term is bounded by $2^{d-1}$, we get 
\begin{equation}\label{trzy}
\begin{split}
\sum_{\boldsymbol{t} \in \Lambda}
\abs{\delta_{\boldsymbol{t}}(U)}
&=\sum_{\boldsymbol{t} \in \Lambda \cap K_{\boldsymbol{t}_0}}
\abs{\delta_{\boldsymbol{t}}(U)} \\
&\leq 2^{d-1} (2m+1)^d\ 2^{-1}
\sum_{\genfrac{}{}{0pt}{1}{\boldsymbol{s,t} \in \Lambda}{\|\boldsymbol{t} -
\boldsymbol{s}\|_{\infty} > m}} I(Z_{\boldsymbol{s}} \neq
0,Z_{\boldsymbol{t}} \neq 0).
\end{split}
\end{equation}
Now (\ref{dwa}) and (\ref{trzy}) together imply (\ref{eM1}).

\subsection*{Case 3. $\text{diam}(\Lambda_0) > 2m$.} \

The assumption implies (\ref{jeden}), hence (\ref{dwa}). Moreover,
for every $\boldsymbol{u}\in \mathbb{Z}^d$ there exists $\boldsymbol{s}\in
\Lambda_0$ such that $\|\boldsymbol{u}-\boldsymbol{s}\|_{\infty}>m$.
From the proof of Case 1 (see (\ref{b})) we know that 
$\delta_{\boldsymbol{t}} (U) \neq 0$ implies that $Z_{\boldsymbol{u}}\ne 0$
for some $\boldsymbol{u} \in B_{\boldsymbol{t}} \setminus
B_{\boldsymbol{t}+\boldsymbol{1}}$ and, under the present assumption, for such
an $\boldsymbol{u}$ there is  $\boldsymbol{s}$ such that 
$Z_{\boldsymbol{s}}\ne 0$ and $\|\boldsymbol{u}-\boldsymbol{s}\|_{\infty}>m$.
Hence we have the following estimates
\[
\begin{split}
\sum_{\boldsymbol{t} \in \Lambda}
\abs{\delta_{\boldsymbol{t}}(U)} &\leq 2^{d-1} \sum_{\boldsymbol{t} \in \Lambda}
\sum_{\boldsymbol{u} \in
(B_{\boldsymbol{t}} \setminus B_{\boldsymbol{t + \boldsymbol{1}}})\cap
\Lambda} 
\sum_{\genfrac{}{}{0pt}{1}{\boldsymbol{s} \in \Lambda}{\|\boldsymbol{s} -
\boldsymbol{u}\|_{\infty} > m}} I(Z_{\boldsymbol{s}} \neq
0,Z_{\boldsymbol{u}} \neq 0)\\
&= 2^{d-1}\sum_{\genfrac{}{}{0pt}{1}{\boldsymbol{s, u} \in \Lambda}
{\|\boldsymbol{s} - \boldsymbol{u}\|_{\infty} > m}} 
\big[\sum_{\boldsymbol{t} \in  \Lambda} 
I(\boldsymbol{u} \in
B_{\boldsymbol{t}} \setminus B_{\boldsymbol{t + \boldsymbol{1}}})\big]
I(Z_{\boldsymbol{s}} \neq 0,Z_{\boldsymbol{u}} \neq 0)\\
&\le 2^{d-1} ((m+1)^d - m^d)\ 
\sum_{\genfrac{}{}{0pt}{1}{\boldsymbol{s, u} \in \Lambda}{\|\boldsymbol{s} -
\boldsymbol{u}\|_{\infty} > m}} I(Z_{\boldsymbol{s}} \neq
0,Z_{\boldsymbol{u}} \neq 0)\\
&\leq  2^{d} (2m + 1)^d \ 2^{-1} 
\sum_{\genfrac{}{}{0pt}{1}{\boldsymbol{s, u} \in \Lambda}{\|\boldsymbol{s} -
\boldsymbol{u}\|_{\infty} > m}} I(Z_{\boldsymbol{s}} \neq
0,Z_{\boldsymbol{u}} \neq 0).
\end{split}
\]
The proof of Theorem \ref{TM} is complete.
\end{proof}
\vskip18pt

\section{An abstract form of the Bonferroni--type inequality }

Consider a family of events ${\mathcal A}=\{A_T\}$ indexed by finite
\underline{subsets} $T$ of ${\mathbb Z}^d$.  A family of events 
${\mathcal C}= \{C_{\boldsymbol{t}}\}$ indexed by \underline{points}
$\boldsymbol{t} \in
\mathbb{Z}^d$ is said to be a {\em complete cover} of $\mathcal A$ \ 
if for every finite sets \ 
$T, T_1, T_2 \subset {\mathbb Z}^d$, \ 
$A_T \subset \bigcup_{\boldsymbol{t}\in T} C_{\boldsymbol{t}}$ \ 
and
\begin{equation}\label{eBSM}
A_{T_1} \triangle A_{T_2}\subset 
\bigcup_{\boldsymbol{t}\in T_1\triangle T_2}  C_{\boldsymbol{t}}.
\end{equation}
Define
$$
\Delta_{\boldsymbol{t}}=\sum_{\varepsilon \in \mathcal{E}}
(-1)^{|\varepsilon|}P(A_{B_{\boldsymbol{t}}^{\varepsilon}}),
$$
where $B_{\boldsymbol{t}}^{\varepsilon}$ is given by (\ref{Be}).
Following the steps of the proof of Theorem \ref{TM} we can prove the
following 
``abstract form'' of the Bonferroni--type inequality. 
\vskip18pt

\begin{theorem}\label{abs}
Let $\mathcal A$ and $\mathcal C$ be families satisfying (\ref{eBSM}).
Then for every finite set $\Lambda \subset {\mathbb Z}^d$
\begin{equation}\label{eA}
\begin{split}
|P(A_{\Lambda})- \sum_{\boldsymbol{t}\in \Lambda}\Delta_{\boldsymbol{t}}|
\le & c_1(d,m) \sum_{\boldsymbol{s}\in \partial\Lambda} P(C_{\boldsymbol{t}})\\
&+c_2(d,m)\sum_{\genfrac{}{}{0pt}{1}{\boldsymbol{s},\boldsymbol{t} \in
\Lambda} 
{\|\boldsymbol{s}-\boldsymbol{t}\|_{\infty}>m}}
P(C_{\boldsymbol{s}}\cap C_{\boldsymbol{t}})
\end{split}
\end{equation}
where constants $c_1$ and $c_2$ are the same as in Theorem \ref{TM}.
\end{theorem} 

\begin{proof} We will only indicate the main steps of the proof.
Since $\Delta_{\boldsymbol{t}}=E\delta_{\boldsymbol{t}}$, where
\begin{equation} \label{da}
\delta_{\boldsymbol{t}}=\sum_{\varepsilon \in \mathcal{E}}
(-1)^{|\varepsilon|}I(A_{B_{\boldsymbol{t}}^{\varepsilon}}), 
\end{equation}
it is enough to prove that
\begin{equation}\label{eM1a}
\begin{split}
\abs{I(A_{\Lambda}) - \sum_{\boldsymbol{t} \in \Lambda}
\delta_{\boldsymbol{t}}} &\leq  c_1(d,m) \sum_{\boldsymbol{s}
\in \partial\Lambda} I(C_{\boldsymbol{s}}) \\
&\quad + c_2(d,m)
\sum_{\genfrac{}{}{0pt}{1}{\boldsymbol{s,t} \in \Lambda}{\|\boldsymbol{t} -
\boldsymbol{s}\|_{\infty} > m}} I(C_{\boldsymbol{s}} \cap 
C_{\boldsymbol{t}}) 
\end{split}
\end{equation}
holds everywhere on the probability space. First we will show that it suffices
to prove (\ref{eM1a}) for the modifications ${\mathcal A'}=\{A_T'\}$
and ${\mathcal C'}=\{C_{\boldsymbol{t}}'\}$ defined as follows
\[
A_T' = A_{T\cap \Lambda}
\]
and
\[
C_{\boldsymbol{t}}' = \begin{cases}
C_{\boldsymbol{t}} & \text{if $\boldsymbol{t}\in \Lambda$},\\
\emptyset & \text{if $\boldsymbol{t} \notin \Lambda$}.
\end{cases}
\]
Note that ${\mathcal C'}$ is a complete cover of ${\mathcal A'}$.
Define $\delta_{\boldsymbol{t}}'$ by replacing in (\ref{da})
$A_{B_{\boldsymbol{t}}^{\varepsilon}}$ with
$A_{B_{\boldsymbol{t}}^{\varepsilon}}'$.  If 
$\delta_{\boldsymbol{t}} \ne \delta_{\boldsymbol{t}}'$, then
(\ref{eBSM}) yields
\[
\begin{split}
|\delta_{\boldsymbol{t}}-\delta_{\boldsymbol{t}}'| 
&\le \sum_{\varepsilon \in \mathcal{E}}
| I(A_{B_{\boldsymbol{t}}^{\varepsilon}}) -
I(A_{B_{\boldsymbol{t}}^{\varepsilon}}') | \\
& \le  \sum_{\varepsilon \in \mathcal{E}} \sum_{\boldsymbol{s}\in 
B_{\boldsymbol{t}}^{\varepsilon}\cap \Lambda^c } I(C_{\boldsymbol{s}})\\
& \le 2^d \sum_{\boldsymbol{s}\in 
B_{\boldsymbol{t}}\cap \Lambda^c } I(C_{\boldsymbol{s}})
\end{split}
\]
which makes the reduction from ${\mathcal A}, {\mathcal C}$ to
${\mathcal A'}, {\mathcal C'}$ possible, analogously to the first part
of the proof of Theorem \ref{TM}, (\ref{eM2})--(\ref{eM3}).
Now we can assume that $C_{\boldsymbol{t}}=\emptyset$ for
$\boldsymbol{t} \notin \Lambda$. Define a random set
\[ \Lambda_0 = \{\boldsymbol{s}\in \mathbb{Z}^d : I(C_{\boldsymbol{t}})=1\}.
\]
(\ref{eM1a}) can now be established by considering the three cases of
$\text{diam}(\Lambda_0)$, exactly as in the proof of Theorem \ref{TM}.
\end{proof}
\vspace{12pt}

Theorem \ref{abs} gives Bonferroni-type inequalities in a variety of
important cases. We mention below some of them. \\[12pt]
\textbf{Examples.}\\[10pt]
\textbf{(i)}\ $\mathcal{C}$ is an arbitrary family of events and 
\[ A_T = \bigcup_{\boldsymbol{t}\in T} C_{\boldsymbol{t}}.\]
In this case Theorem \ref{abs} generalizes the classical Bonferroni
inequality (\ref{e1}).\\[8pt]
\textbf{(ii)}\ $A_T = \{\max_{\boldsymbol{t}\in T}
Z_{\boldsymbol{t}} > \lambda\}$ and $C_{\boldsymbol{t}} 
= \{Z_{\boldsymbol{t}} > \lambda\}$, where $\{Z_{\boldsymbol{t}}\}$ is a
real-valued random field. This 
is a special important case of (i).\\[8pt]
\textbf{(iii)}\ $A_T = \{\sum_{\boldsymbol{t}\in T}
Z_{\boldsymbol{t}} \in U\}$, where $Z_{\boldsymbol{t}}$ 
and $U$ are as in Section 2, and $C_{\boldsymbol{t}} =
\{Z_{\boldsymbol{t}} \ne 0\}$. This shows that Theorem \ref{TM}
is a special case of Theorem \ref{abs}.\\[8pt]
\textbf{(iv)}\ $A_T = \{\prod_{\boldsymbol{t}\in T}
Z_{\boldsymbol{t}} \in U\}$ and $C_{\boldsymbol{t}} = 
\{Z_{\boldsymbol{t}} \ne 1\}$, where $\{Z_{\boldsymbol{t}}\}$ is a
complex-valued random field and $U$  
is a Borel subset of the complex plane such that $1 \notin U$.\\[8pt]
\textbf{(v)}\ Let $\{Z_{\boldsymbol{t}}\}$ be a random field
taking values in a 
measurable semigroup $G$ with the neutral element $I$. Fix a linear order in
$\mathbb{Z}^d$ to avoid the ambiguity in the definition of
$\Pi_T = \prod_{\boldsymbol{t}\in T} Z_{\boldsymbol{t}}$ in the case when $G$
is non Abelean (for instance, the lexicographical order). Then $A_T = \{\Pi_T
\in U\}$ and
$C_{\boldsymbol{t}} = \{Z_{\boldsymbol{t}} \ne I\}$ satisfy the assumptions
of Theorem \ref{abs}, provided $I \notin U$. In particular, Theorem
\ref{abs} gives the Bonferroni-type inequality for products of random
matrices. 
\vskip18pt

\bibliographystyle{amsalpha}

\end{document}